\documentclass{amsart}
\usepackage{amsmath, amssymb, verbatim, amscd, amsthm, bm}
\usepackage[normalem]{ulem}
\usepackage[driverfallback=dvipdfm]{hyperref}
\usepackage[dvipdfm]{graphicx}
\usepackage{color,graphicx,shortvrb}
\usepackage{enumitem}
\usepackage{tikz}
\usepackage{xcolor}
\usepackage{mathtools}

\newtheorem{thm}{Theorem}

\newtheorem{lem}{Lemma}[section] 
\newtheorem{prop}[lem]{Proposition}
\newtheorem{cor}[lem]{Corollary}

\theoremstyle{definition}

\newtheorem{rem}{Remark}[section]

\numberwithin{equation}{section}

\newcommand{\q}{\quad} 
\newcommand{\qq}{\qquad} 
\newcommand{\qbox}[1]{\q \mbox{#1} \q}

\newcommand{\C}{\mathbb C} 
\newcommand{\N}{\mathbb N} 
\newcommand{\T}{\mathbb T}
\newcommand{\al}{\alpha}

\renewcommand{\Re}{\textrm{Re}}
\newcommand{\set}[1]{\left\{ #1  \right\}}

\newcommand{\fr}[2]{\frac{#1}{#2}}
\newcommand{\SA}{S_\A}
\newcommand{\SB}{S_\B}
\newcommand{\TT}{\widetilde{T}}
\newcommand{\A}{A}
\newcommand{\B}{B}
\newcommand{\cha}{\mathrm{Ch}(A)}
\newcommand{\chb}{\mathrm{Ch}(B)}

\newcommand{\numx}{\nu\mx}
\newcommand{\Vinf}[1]{\left\Vert #1 \right\Vert_\infty} 

\renewcommand{\th}{\theta}
\newcommand{\thx}{\theta_x}
\newcommand{\thp}{\theta_{p}}
\newcommand{\pa}{\partial\A}
\newcommand{\pb}{\partial\B}
\newcommand{\Gam}{\Gamma}
\newcommand{\MA}{\mathcal{M}_{\A}}
\newcommand{\MB}{\mathcal{M}_{\B}}

\newcommand{\Del}{\Delta}
\newcommand{\la}{\lambda}
\newcommand{\mx}{M_x}
\newcommand{\Mp}{M_p}
\newcommand{\lamx}{\la M_x}
\newcommand{\nump}{\nu\Mp}
\newcommand{\nphx}{N_{\sigma(x)}}
\newcommand{\nphp}{N_{\sigma(p)}}
\newcommand{\Tx}{T_x}
\newcommand{\Tp}{T_p}
\newcommand{\ov}{\overline}
\newcommand{\X}{X}
\newcommand{\wt}{\widetilde}
\newcommand{\alx}{\al_x}
\newcommand{\alxp}{\al_x'}
\newcommand{\alxt}{\wt{\al}_x}
\newcommand{\unit}{\bm{1}}
\newcommand{\ve}{\varepsilon}

 \usepackage{marginnote}
 \usepackage{geometry}

\begin{document}


\title[phase-isometries on uniform algebras]
{Globalization of local sign structures
for phase-isometries on uniform algebras}

\author[Y.~Enami]{Yuta Enami}
\address[Y.~Enami]
{Graduate School of Science and Technology,
Niigata University, Niigata 950-2181, Japan}
\email{f21j008j@mail.cc.niigata-u.ac.jp}

\author[D.~Hirota]{Daisuke Hirota}
\address[D.~Hirota]
{National Institute of Technology,
Tsuruoka College,
Yamagata 997-8511, Japan}
\email{dhirota@tsuruoka-nct.ac.jp}

\author[I.~Matsuzaki]{Izuho Matsuzaki}
\address[I.~Matsuzaki]
{Graduate School of Science and Technology,
Niigata University, Niigata 950-2181, Japan}
\email{matsuzaki@m.sc.niigata-u.ac.jp}

\author[T.~Miura]{Takeshi Miura}
\address[T.~Miura]
{Department of Mathematics,
Faculty of Science, Niigata University,
Niigata 950-2181, Japan}
\email{miura@math.sc.niigata-u.ac.jp}


\begin{abstract}
We study surjective phase-isometries between the unit spheres of
uniform algebras.  Although such maps preserve maximal convex sets up
to signs, the resulting local sign ambiguity prevents a direct
application of the usual Banach--Stone type arguments for isometries.
The main point of the paper is to prove that these local sign
structures can be globalized on the Choquet boundary.  To this end, we
refine an additive Bishop-type construction and use it to propagate
the sign information among the maximal convex sets associated with
boundary points.

As a consequence, every surjective phase-isometry admits a boundary
representation by means of a global sign function, a unimodular
weight, a homeomorphism between the Choquet boundaries, and a
clopen decomposition into complex-linear and conjugate-linear parts.
We then extend this representation to the maximal ideal spaces
and obtain the corresponding real-algebraic Banach--Stone type
representation.
\end{abstract}

\maketitle

\section{Introduction}

The study of linear and nonlinear isometries 
between Banach spaces has long been a central topic 
in functional analysis since the publication of Banach's
famous book \cite{bana} in 1932.
Among the most influential problems in this area is
Tingley's problem \cite{ting}, 
which asks whether every surjective isometry 
between the unit spheres of two Banach spaces 
extends to a linear isometry between the spaces themselves. 
Research surrounding this problem has revealed 
that the metric geometry of the unit spheres 
often encodes deep linear, algebraic, and topological
structures of the underlying spaces.

In 1931, Wigner considered
a map $T\colon H\to H$ on a complex Hilbert space $H$
with the following property:
\[
|\langle T(f),T(g)\rangle|
=|\langle f,g\rangle|
\qq(f,g\in H),
\]
where $\langle\cdot,\cdot\rangle$
denotes the inner product on $H$.
Wigner \cite{wign} proved that there exists a function
$\th\colon H\to\C$ such that $|\th(f)|=1$ for all $f\in H$
and $\th T$, which maps each $f\in H$ to $\th(f)T(f)$,
is linear or conjugate linear isometry.
This result may be viewed as a precursor to modern studies
of phase-isometries.
The theories of isometries and phase-isometries developed
largely independently, although both are closely connected with
the problem of recovering algebraic and topological structures
from geometric information on the unit spheres.
In recent years, phase-isometries have attracted increasing attention
in connection with this geometric viewpoint.
Most existing studies on phase-isometries
have been carried out in the setting of real Banach spaces
(see, for example, \cite{hiro,ilis,lita,tan1,tan2,tan3,tan4}).
In contrast, the complex case remains much less understood,
except for a few recent works
(see, for example, \cite{huan1, liu1}).
The presence of conjugate structures and local sign ambiguity
creates additional difficulties in the complex setting.

Let $E$ and $F$ be real or complex Banach spaces. 
A mapping $T\colon E\to F$ is called a phase-isometry if
\[
\{\|T(f)+T(g)\|,\|T(f)-T(g)\|\}
=
\{\|f+g\|,\|f-g\|\}
\]
for all $f,g\in E$.
Phase-isometries may be viewed as
Banach-space analogues of Wigner-type transformations.
Maksa and P\'{a}les \cite{maks} proved that, if $E$ and $F$
are Hilbert spaces, then the map $T$ satisfies
the norm condition if and only if the inner product condition
in the result of Wigner.
Hence the norm condition of phase-isometries provides
a natural Banach-space formulation of Wigner-type phenomena.
It is known that surjective phase-isometries preserve
maximal convex subsets up to signs
(see \cite[Proposition~2.5]{tan4} and Lemma~\ref{lem2.2}).
One might expect standard geometric arguments
for isometries to apply to phase-isometries.

However, this expectation turns out to be misleading.
Because of the local sign ambiguity
of maximal convex subsets for phase-isometries,
one cannot directly apply
geometric arguments for isometries.
One needs to globalize the local sign structures
associated with maximal convex subsets
to obtain structural representations of phase-isometries.
For this purpose, the existing additive Bishop-type lemma
is insufficient in its current form.

The main difficulty is that the local signs
arising from maximal convex subsets
cannot be chosen independently
at different boundary points.
To overcome this obstruction, we refine
an additive Bishop-type construction so that
the associated local sign information can be
propagated among maximal convex subsets.
The globalization of these local sign structures
plays a crucial role throughout the paper.

In this paper, we determine
the precise structure of surjective phase-isometries
on the unit spheres of uniform algebras.
More specifically, we first characterize
such mappings as weighted composition operators
on the Choquet boundaries with a global sign function.
Using this representation, we show that
every surjective phase-isometry on the unit sphere
extends to a surjective phase-isometry
between the corresponding uniform algebras
on the Shilov boundaries.
We then prove that this extension induces
a real-linear algebra isomorphism,
which leads to a representation theorem
on the maximal ideal spaces.
Consequently, every surjective phase-isometry
considered in this paper is described by
a homeomorphism between maximal ideal spaces,
combined with a unimodular function and,
possibly, complex conjugation.
Conversely, every mapping of the form
in the Main Theorem is a surjective phase-isometry.

The paper is organized as follows.
In Section~\ref{sect2}, we recall basic facts on uniform algebras,
Choquet boundaries, and phase-isometries, and we state
the Main Theorem.
Section~\ref{sect3} is devoted to several preliminary results concerning
maximal convex sets and their behavior under surjective
phase-isometries.
In Section~\ref{sect4}, we analyze the local phase structure associated with a
surjective phase-isometry.  After establishing the consistency of the
local sign ambiguity on suitable subsets of the Choquet boundary, we
derive a boundary representation formula by combining this analysis
with an additive Bishop-type construction.  This leads to the proof of
the Main Theorem.

\section{Preliminaries and main theorem}
\label{sect2}

\subsection{Preliminaries}

We collect several known facts and fix notation
used throughout the paper.

Let $X$ be a compact Hausdorff space.
A uniform algebra on $X$ is a uniformly closed subalgebra
$A$ of $C(X)$ containing the constant functions and separating
the points of $X$.
Let $\A$ and $\B$ be uniform algebras 
on compact Hausdorff spaces $X$ and $Y$, respectively.
We denote the unit spheres of $A$ and $B$ by
$S_A$ and $S_B$, respectively.
We denote by $\cha$ and $\chb$ the Choquet
boundaries of $A$ and $B$, respectively, and by $\pa$ and
$\pb$ their Shilov boundaries.  Recall that $\cha$ is
a boundary for $A$ and dense in $\pa$.
In particular, $\pa$ is a boundary for $A$;
the analogous facts hold for $B$. 
We shall also use the following standard peak property of
Choquet boundary points:
\begin{equation}\label{peak}
\begin{lgathered}
\text{for every $x\in\cha$,
every open neighborhood $U$ of $x$,
and every $\varepsilon>0$,}\\
\text{there exists $f\in S_A$
such that
$f(x)=1$ and $|f|<\varepsilon$
on $X\setminus U$.}
\end{lgathered}
\end{equation}
For background on uniform algebras and related boundary theory,
we refer the reader to \cite{brow} and \cite{hatori1}.

We denote the maximal ideal spaces
of $A$ and $B$ by $\MA$
and $\MB$, respectively,
and $\widehat f$ denotes the
Gelfand transform of $f$.

Throughout the paper,
$T\colon \SA\to\SB$
denotes a surjective phase-isometry,
that is,
\[
\set{\Vinf{T(f)+T(g)},\Vinf{T(f)-T(g)}}
=\set{\Vinf{f+g},\Vinf{f-g}}
\qq(f,g\in\SA).
\]

\subsection{Main Theorem}

The following theorem gives the precise structure
of surjective phase-isometries between the unit spheres
of uniform algebras.

\begin{thm}[Main Theorem]\label{mainthm}
Let $A$ and $B$ be uniform algebras on $X$ and $Y$, 
respectively, and let
$T\colon S_A\to S_B$ be a surjective phase-isometry.
Then there exist
\begin{itemize}
\item
a mapping $\theta\colon S_A\to\{\pm1\}$
satisfying $\theta(-f)=\theta(f)$ for $f\in S_A$,

\item
an invertible element $\Gamma\in B$ satisfying
$|\Gamma|=1$ on $\pb$,

\item
a real-linear algebra isomorphism
$\TT\colon A\to B|_{\pb}$.
\end{itemize}
Here $B|_{\pb}$ denotes the restriction algebra
of $B$ on $\pb$.

These objects satisfy
\begin{equation}\label{main-eq}
T(f)|_{\pb}
=\theta(f)\Gamma|_{\pb}\cdot\TT(f)
\qquad (f\in S_A).
\end{equation}

Moreover, there exist a homeomorphism
$\varphi\colon\pb\to\pa$
and a clopen subset $Y_s$ of $\pb$ such that
\begin{equation}\label{thm1.2}
\TT(f)(y)
=
\begin{cases}
f(\varphi(y)), & y\in Y_s,\\[1mm]
\overline{f(\varphi(y))}, & y\in\pb\setminus Y_s,
\end{cases}
\qquad (f\in A,\ y\in\pb).
\end{equation}

Furthermore, the representation formula
\eqref{thm1.2} induces
a Banach--Stone-type representation
on the maximal ideal spaces
in the real-algebraic setting.
Let $\Lambda\colon B|_{\pb}\to B$ denote the inverse of the
restriction map, and set
$\mathcal{T}=\Lambda\circ\TT\colon A\to B$.
Moreover, there exist a homeomorphism
$\Phi\colon\MB\to\MA$
and a clopen subset $Y_m$ of $\MB$ such that
\[
\Phi(\rho_y)=\tau_{\varphi(y)}
\qquad (y\in\pb),
\]
and
\[
\widehat{\mathcal{T}(f)}(\rho)
=
\begin{cases}
\widehat f(\Phi(\rho)), & \rho\in Y_m,\\[1mm]
\overline{\widehat f(\Phi(\rho))},
& \rho\in\MB\setminus Y_m,
\end{cases}
\qquad (f\in A,\ \rho\in\MB).
\]
Here $\rho_y$ and $\tau_x$ denote the point evaluations at
$y\in\pb$ and $x\in\pa$, respectively.
\end{thm}

\begin{rem}
Let $T\colon S_A\to S_B$ be a mapping
satisfying
\[
T(f)|_{\pb}=\theta(f)\Gamma|_{\pb}\cdot\TT(f)
\qq(f\in\SA),
\]
where $\theta\colon S_A\to\{\pm1\}$ satisfies
$\theta(-f)=\theta(f)$,
$\Gamma$ satisfies $|\Gam|=1$ on $\pb$,
and $\TT$ is a real-linear algebra isomorphism
of the form described in Theorem~\ref{mainthm}.
Then $T$ is a surjective phase-isometry.
Hence Theorem~\ref{mainthm}
completely characterizes surjective phase-isometries
between the unit spheres of uniform algebras.
\end{rem}

\subsection{Geometric lemmas}

Lemma~\ref{lem2.1} below shows that
$T$ is injective and that $T^{-1}$ is again
a surjective phase-isometry.
Since all assumptions are symmetric
with respect to $T$ and $T^{-1}$, 
we mainly analyze $T$, while the same arguments 
apply to $T^{-1}$ whenever needed.
For convenience, we write
\[
\pm M=M\cup(-M)
\qquad (M\subset\SA).
\]

We shall repeatedly use the following basic facts.

\begin{lem}\label{lem2.1}
The mapping $T$ is injective and odd, that is,
\[
T(-f)=-T(f)
\qquad (f\in\SA),
\]
and its inverse $T^{-1}$ is again a surjective phase-isometry.
\end{lem}

\begin{proof}
See \cite[Lemma~2]{tan1}.
\end{proof}

The next lemma plays a key role
in the analysis of phase-isometries.
At the same time, this preservation property
contains an intrinsic sign ambiguity,
which is the main difficulty of the arguments
throughout the paper.

\begin{lem}\label{lem2.2}
For each maximal convex subset $M$ of $\SA$,
there exists a maximal convex subset $N$ of $\SB$ such that
\[
T(\pm M)=\pm N.
\]
\end{lem}

\begin{proof}
See \cite[Proposition~2.5]{tan4}.
\end{proof}

We denote by $\T$ the unit circle in $\mathbb C$.

For each $(\lambda,x)\in\T\times\cha$
and $(\mu,y)\in\T\times\chb$, define
\[
\lambda M_x
=
\{f\in S_A:f(x)=\lambda\},
\qquad
\mu N_y
=
\{u\in S_B:u(y)=\mu\}.
\]
These subsets describe the local phase geometry
associated with Choquet boundary points.

The next lemma characterizes
all maximal convex subsets of $\SA$.
We write $\operatorname{ext}(A_1^*)$ for the set
of extreme points of the dual unit ball.

\begin{lem}\label{lem2.3}
A subset $M$ of $S_A$ is a maximal convex subset
if and only if
\[
M=\lambda M_x
\]
for some $(\lambda,x)\in\T\times\cha$.
In particular, $\lambda M_x$ is a maximal convex subset
of $S_A$ for every
$(\lambda,x)\in\T\times\cha$.
\end{lem}

\begin{proof}
Let $M$ be a maximal convex subset of $S_A$.
By \cite[Lemma~3.1]{hatori4}, there exists
$\xi\in\operatorname{ext}(A_1^*)$ such that
$M=\xi^{-1}(1)\cap S_A$.
By the Arens--Kelley theorem (see, for example,
\cite[Corollary~2.3.6]{flem} or \cite[p.~41]{hatori1}),
\[
\operatorname{ext}(A_1^*)
=
\{\omega\tau_x:\omega\in\T,\
x\in\cha\}.
\]
Hence $\xi=\omega\tau_x$ for some
$(\omega,x)\in\T\times\cha$,
and therefore
\[
M=(\omega\tau_x)^{-1}(1)\cap S_A
=\{f\in S_A:\omega f(x)=1\}
=\overline{\omega}M_x.
\]

Conversely, let $(\lambda,x)\in\T\times\cha$.
It is immediate that $\lambda M_x$ is convex.
Choose a maximal convex subset $F$ of $S_A$ with $\lambda M_x\subset F$.
By the first part, there exists $(\nu,p)\in\T\times\cha$
such that $F=\nu M_p$. Since $\lambda M_x\subset\nu M_p$,
Lemma~\ref{lem3.1} below implies $(\lambda,x)=(\nu,p)$. Thus
$\lambda M_x=F$, and hence $\lambda M_x$ is maximal.
\end{proof}

We also need the following auxiliary consequence
of additive Bishop-type arguments.

\begin{lem}\label{lem2.4}
Let $f\in\A$ and $x\in\cha$.
If $f(x)\neq0$, then there exists $g\in M_x$ such that
\[
\frac{fg}{f(x)}\in M_x.
\]
\end{lem}

\begin{proof}
This follows immediately from \cite[Lemma~2.3]{hatori3}.
\end{proof}

\section{Maximal convex subsets and local sign behavior}
\label{sect3}

We now begin the geometric analysis of surjective phase-isometries
via maximal convex subsets of unit spheres.
Our first objective is to localize the image of each family
$\pm\lambda M_x$ under $T$
and to show that the associated boundary point on $Y$
depends only on $x$.
We denote by $\unit$ the constant function
with value $1$.

Although elementary, the following result is crucial
in our arguments.

\begin{lem}\label{lem3.1}
Let $(\lambda,x),(\nu,p)\in\T\times\cha$.
If $\lambda M_x\subset\pm\nu M_p$,
then $x=p$.
In particular, if
$\lambda M_x\subset\nu M_p$,
then $(\lambda,x)=(\nu,p)$.
\end{lem}

\begin{proof}
Assume that $\lambda M_x\subset\pm\nu M_p$
and suppose that $x\neq p$.
Choose $f_0\in\A$ satisfying $f_0(x)\neq f_0(p)$, and define
\[
f=\frac{f_0-f_0(p)\unit}{f_0(x)-f_0(p)}.
\]
Then $f(x)=1$ and $f(p)=0$.
By Lemma~\ref{lem2.4}, there exists $g\in M_x$
such that $fg\in M_x$.
Put $h=fg$.
Then $h(x)=1$ and $h(p)=0$.
Hence $\lambda h\in\lambda M_x$,
whereas $\lambda h\notin\pm\nu M_p$,
since every element of $\pm\nu M_p$
has modulus one at $p$.
This contradiction shows that $x=p$.

If $\lambda M_x\subset\nu M_p$,
then $x=p$ by the first part,
and hence $\lambda M_x\subset\nu M_x$.
Since $\lambda\unit\in\lambda M_x$,
we obtain $\lambda\unit\in\nu M_x$,
which implies $\lambda=\nu$.
\end{proof}

The next lemma is the starting point
for constructing the boundary correspondence
associated with $T$.
At this stage, the image of each element
$f\in\lambda M_x$
is determined only up to sign within
$\pm\mu N_y$.

\begin{lem}\label{lem3.2}
For each $(\lambda,x)\in\T\times\cha$,
there exists $(\mu,y)\in\T\times\chb$
such that
\[
T(\pm\lambda M_x)=\pm\mu N_y.
\]
\end{lem}

\begin{proof}
Fix $(\lambda,x)\in\T\times\cha$.
By Lemma~\ref{lem2.3},
$\lambda M_x$ is a maximal convex subset of $\SA$.
Hence Lemma~\ref{lem2.2} yields
a maximal convex subset $N$ of $\SB$ satisfying
$T(\pm\lambda M_x)=\pm N$.
Applying Lemma~\ref{lem2.3} to $\SB$,
we obtain $(\mu,y)\in\T\times\chb$
such that $N=\mu N_y$.
Therefore, $T(\pm\lambda M_x)=\pm\mu N_y$.
\end{proof}

To prove that the point $y$ above depends only on $x$,
we need the following separation lemma.
This will allow us to define
the boundary correspondence associated with $T$.

\begin{lem}\label{lem3.3}
Let $\mu,\zeta\in\T$ and let $y,q\in\chb$ with $y\neq q$.
Then there exist $u\in\mu N_y$ and $v\in\zeta N_q$
such that
\[
\|u+v\|_\infty<\sqrt2,\qq
\|u-v\|_\infty<\sqrt2.
\]
\end{lem}

\begin{proof}
Choose disjoint open subsets $U,V$ of $Y$
with $y\in U$ and $q\in V$.
Since $y,q\in\chb$,
by the peak property \eqref{peak},
there exist $u\in\mu N_y$ and $v\in\zeta N_q$
such that
\[
|u|<\frac13 \text{ on } Y\setminus U,
\qquad
|v|<\frac13 \text{ on } Y\setminus V.
\]
If $z\in U$, then $z\notin V$, so
\[
|u(z)\pm v(z)|
\le |u(z)|+|v(z)|
<1+\frac13<\sqrt2.
\]
If $z\in Y\setminus U$, then similarly
$|u(z)\pm v(z)|<\sqrt2$.
Hence $\|u+v\|_\infty<\sqrt2$
and $\|u-v\|_\infty<\sqrt2$.
\end{proof}

The preceding lemma shows that,
although the unimodular constants in Lemma~\ref{lem3.2}
may vary with $\lambda$,
the associated boundary point is in fact uniquely determined.

\begin{lem}\label{lem3.4}
For each $\lambda,\nu\in\T$ and $x\in\cha$,
let $(\mu,y),(\zeta,q)\in\T\times\chb$
satisfy
\[
T(\pm\lambda M_x)=\pm\mu N_y,
\qquad
T(\pm\nu M_x)=\pm\zeta N_q.
\]
Then $y=q$.
\end{lem}

\begin{proof}
Assume to the contrary that $y\neq q$.
Choose $u\in\mu N_y$ and $v\in\zeta N_q$
as in Lemma~\ref{lem3.3}, and set
\[
f=T^{-1}(u),\qq
g=T^{-1}(v).
\]
Then $f\in\pm\lambda M_x$ and $g\in\pm\nu M_x$
by assumption.
Hence there exist $a,b\in\{\pm1\}$ such that
$f(x)=a\lambda$ and $g(x)=b\nu$.
Since $T^{-1}$ is a phase-isometry, we have
\[
\{\|f+g\|_\infty,\|f-g\|_\infty\}
=\{\|u+v\|_\infty,\|u-v\|_\infty\}.
\]
By Lemma~\ref{lem3.3}, both norms on the right are smaller than
$\sqrt2$. Thus
\[
\|f+g\|_\infty<\sqrt2,
\qquad
\|f-g\|_\infty<\sqrt2.
\]
Evaluating at $x$, with the same choice of sign on both sides, gives
\[
|a\lambda\pm b\nu|=|f(x)\pm g(x)|\le \|f\pm g\|_\infty.
\]
Equivalently,
$|1+ab\overline{\lambda}\nu|<\sqrt2$ and
$|1-ab\overline{\lambda}\nu|<\sqrt2$.
This is impossible, since for every $\omega\in\T$,
at least one of
$|1+\omega|$ and $|1-\omega|$ is not smaller than $\sqrt2$.
Hence $y=q$.
\end{proof}

\section{Local sign structures and their globalization}
\label{sect4}

We now analyze the local sign ambiguity
arising in Lemma~\ref{lem3.2}.
For each $(\lambda,x)\in\T\times\cha$,
the image of $\lambda M_x$ under $T$
is determined only up to sign within
$\pm\mu N_y$.
To normalize this ambiguity,
we first examine the behavior of the constant functions
$\lambda\unit\in\lambda M_x$.

The main idea is that,
if the representation formulas
\eqref{main-eq} and \eqref{thm1.2}
in the Main Theorem are expected to hold,
then the image of $\lambda\unit$
should behave essentially like
$\lambda$ or $\overline\lambda$
up to a unimodular factor.
This leads to the analysis of the maps
$\alpha_x$ and their normalized forms.
Once these maps are understood,
the local sign ambiguity for general elements
$f\in\lambda M_x$
can be controlled relative to the image of $\lambda\unit$.

The remaining task is to globalize
these local sign normalizations.
For this purpose,
the additive Bishop-type construction
must be refined quantitatively.
This refinement allows the local phase information
to propagate among maximal convex subsets.

\subsection{Boundary correspondences and constant functions}

For each $x\in\cha$, let $\sigma(x)\in\chb$
denote the unique point as in Lemma~\ref{lem3.4}
satisfying the following property:
for every $\lambda\in\T$, there exists
$\mu\in\T$ such that
\[
T(\pm\lambda M_x)=\pm\mu N_{\sigma(x)}.
\]
We first introduce the scalar map obtained
by evaluating $T(\lambda\unit)$ at $\sigma(x)$.

For each fixed $x\in\cha$,
define a map
$\alpha_x\colon\T\to\T$ by
\begin{equation}\label{eq4.1}
\alpha_x(\lambda)=T(\lambda\unit)(\sigma(x))
\qquad (\lambda\in\T).
\end{equation}
Since $\lambda\unit\in\lambda M_x$,
the defining property of $\sigma(x)$ gives
\[
T(\lambda\unit)\in\pm\mu N_{\sigma(x)}.
\]
Hence $\alpha_x(\lambda)\in\{\pm\mu\}$.
Combining this with
$T(\pm\lamx)=\pm\mu\nphx$,
we obtain
\begin{equation}\label{eq4.2}
T(\pm\lambda M_x)=\pm\alpha_x(\lambda)N_{\sigma(x)}
\qquad (\lambda\in\T,\ x\in\cha).
\end{equation}

Applying the same argument to $T^{-1}$,
for each $y\in\chb$ there exists a unique point $\phi(y)\in\cha$
such that
\[
T^{-1}(\pm\mu N_y)=\pm\beta_y(\mu)M_{\phi(y)}
\]
for some scalar map
$\beta_y\colon\T\to\T$.
That is,
\begin{equation}\label{eq4.3}
T^{-1}(\pm\mu N_y)=\pm\beta_y(\mu)M_{\phi(y)}
\qquad (\mu\in\T,\ y\in\chb).
\end{equation}

Thus we obtain mappings
\[
\sigma\colon\cha\to\chb,
\qquad
\phi\colon\chb\to\cha.
\]
We now show that these two boundary correspondences
are mutually inverse.

\begin{lem}\label{lem4.1}
The maps
$\sigma\colon\cha\to\chb$ and
$\phi\colon\chb\to\cha$
are bijections satisfying
$\phi=\sigma^{-1}$.
\end{lem}

\begin{proof}
Fix $x\in\cha$ and $\lambda\in\T$. Set
$\mu=\alpha_x(\lambda)$ and $y=\sigma(x)$.
By \eqref{eq4.2},
\[
\lambda M_x
\subset
\pm\lambda M_x
=
T^{-1}(\pm\mu N_y).
\]
Equality \eqref{eq4.3} gives
$\lamx\subset\pm\beta_y(\mu)M_{\phi(y)}$.
Lemma~\ref{lem3.1} yields $x=\phi(y)=\phi(\sigma(x))$, and hence
$\phi\circ\sigma=\mathrm{id}_{\cha}$,
where $\mathrm{id}_{\cha}$ denotes the identity map on $\cha$.
Applying the same argument to $T^{-1}$ gives
$\sigma\circ\phi=\mathrm{id}_{\chb}$. Therefore $\phi=\sigma^{-1}$.
\end{proof}

For each $x\in\cha$, we normalize $\alpha_x$ by setting
\begin{equation}\label{eq4.4}
\alpha_x'(\lambda)=\overline{\alpha_x(1)}\,\alpha_x(\lambda)
\qquad (\lambda\in\T).
\end{equation}
Then $\alpha_x'(1)=1$.

The next step is to determine the possible form of $\alpha_x'$.
First, we reduce the defining property of phase-isometry $T$
to the mapping $\alpha_x'\colon\T\to\T$.

\begin{lem}\label{lem4.2}
For each $x\in\cha$ and $\lambda,\nu\in\T$,
\[
\{|\alpha_x'(\lambda)+\alpha_x'(\nu)|,\,
|\alpha_x'(\lambda)-\alpha_x'(\nu)|\}
=\{|\lambda+\nu|,\,
|\lambda-\nu|\}.
\]
\end{lem}

\begin{proof}
By \eqref{eq4.1},
$\alpha_x(\lambda)=T(\lambda\unit)(\sigma(x))$ and
$\alpha_x(\nu)=T(\nu\unit)(\sigma(x))$.
Put
\[
a=|\alpha_x(\lambda)+\alpha_x(\nu)|,
\qquad
b=|\alpha_x(\lambda)-\alpha_x(\nu)|,
\]
and
\[
c=\|T(\lambda\unit)+T(\nu\unit)\|_\infty,
\qquad
d=\|T(\lambda\unit)-T(\nu\unit)\|_\infty.
\]
Then $a\le c$ and $b\le d$. Since $T$ is a phase-isometry,
\[
\{c,d\}=\{|\lambda+\nu|,\ |\lambda-\nu|\}.
\]
Hence $c^2+d^2=|\lambda+\nu|^2+|\lambda-\nu|^2=4$.
On the other hand, since $\alpha_x(\lambda)$ and $\alpha_x(\nu)$ are
unimodular, $a^2+b^2=4$.
Because $a,b,c,d\ge0$ with
$a\le c$ and $b\le d$,
the equality $a^2+b^2=c^2+d^2$ forces
$a=c$ and $b=d$. Thus
\[
\{a,b\}=\{|\lambda+\nu|,\ |\lambda-\nu|\}.
\]
Multiplication by the unimodular scalar $\overline{\alpha_x(1)}$ gives
the desired assertion for $\alpha_x'$.
\end{proof}

The next result gives a criterion to determine
the value of $\alpha_x'(\lambda)$.

\begin{lem}\label{lem4.3}
Let $\lambda,\omega\in\T$.
If $\{|\omega+1|,\ |\omega-1|\}
=\{|\lambda+1|,\ |\lambda-1|\}$,
then
$\omega\in\{\pm\lambda,\pm\overline{\lambda}\}$.
\end{lem}

\begin{proof}
Squaring the two elements of the unordered sets gives
\[
\{2+2\Re\,\omega,\ 2-2\Re\,\omega\}
=
\{2+2\Re\lambda,\ 2-2\Re\lambda\}.
\]
Hence $\Re\,\omega\in\set{\pm\Re\lambda}$.

Since $\lambda,\omega\in\T$, the real part determines the point
up to complex conjugation. 
Thus, if $\Re\,\omega=\Re\lambda$, then
$\omega\in\{\lambda,\overline{\lambda}\}$,
while if $\Re\,\omega=-\Re\lambda$, then
$\omega\in\{-\lambda,-\overline{\lambda}\}$.
Hence
$\omega\in\{\pm\lambda,\pm\overline{\lambda}\}$.
\end{proof}

The following result is the first step to prove that
$\alpha_x'$ behaves like identity or conjugation.

\begin{cor}\label{cor4.4}
For each $x\in\cha$ and $\lambda\in\T$,
\[
\alpha_x'(\lambda)\in\{\pm\lambda,\pm\overline{\lambda}\}.
\]
\end{cor}

\begin{proof}
By Lemma~\ref{lem4.2}, applied with $\nu=1$, and since
$\alpha_x'(1)=1$, we have
\[
\{
|\alpha_x'(\lambda)+1|,\,
|\alpha_x'(\lambda)-1|
\}
=
\{|\lambda+1|,\,|\lambda-1|\}.
\]
The assertion follows from Lemma~\ref{lem4.3}.
\end{proof}

Now we show that the local information on $\alpha_x'$
given by Corollary~\ref{cor4.4} extends globally.

\begin{lem}\label{lem4.5}
Let $x\in\cha$ and let
$\nu_0\in\T\setminus\{\pm1,\pm i\}$.
Then the following assertions hold.

\begin{enumerate}
\item[(i)]
If $\alpha_x'(\nu_0)\in\{\pm\nu_0\}$, then
$\alpha_x'(\lambda)\in\{\pm\lambda\}$
for all $\lambda\in\T$.

\item[(ii)]
If $\alpha_x'(\nu_0)\in\{\pm\overline{\nu_0}\}$, then
$\alpha_x'(\lambda)\in\{\pm\overline{\lambda}\}$
for all $\lambda\in\T$.
\end{enumerate}
\end{lem}

\begin{proof}
By Corollary~\ref{cor4.4},
\[
\alpha_x'(\nu_0)
\in
\{\pm\nu_0,\pm\overline{\nu_0}\}.
\]
We prove (i). The proof of (ii) is analogous.

Let $\lambda\in\T$ be arbitrary, and write
$\lambda=\omega\nu_0$ with $\omega\in\T$.
Set $\mu=\alpha_x'(\nu_0)$. By assumption,
$\mu\in\{\pm\nu_0\}$.

Applying Lemma~\ref{lem4.2} to
$\lambda$ and $\nu_0$, we obtain
\[
\{
|\alpha_x'(\lambda)+\mu|,\,
|\alpha_x'(\lambda)-\mu|
\}
=
\{
|\lambda+\nu_0|,\,
|\lambda-\nu_0|
\}
=
\{
|\omega+1|,\,
|\omega-1|
\}.
\]
Since $|\mu|=1$, this is equivalent to
\[
\{
|\overline{\mu}\alpha_x'(\lambda)+1|,\,
|\overline{\mu}\alpha_x'(\lambda)-1|
\}
=
\{
|\omega+1|,\,
|\omega-1|
\}.
\]
Hence Lemma~\ref{lem4.3} yields
$\overline{\mu}\alpha_x'(\lambda)
\in\{\pm\omega,\pm\overline{\omega}\}$.
Since $\mu\in\{\pm\nu_0\}$, it follows that
\[
\alpha_x'(\lambda)\in\{\pm\omega\nu_0,\pm\overline{\omega}\nu_0\}.
\]
On the other hand, Corollary~\ref{cor4.4} gives
\[
\alpha_x'(\lambda)\in\{\pm\lambda,\pm\overline{\lambda}\}
=
\{\pm\omega\nu_0,\pm\overline{\omega}\,\overline{\nu_0}\}.
\]
Because $\nu_0\notin\{\pm1,\pm i\}$, we have
$\{\pm\overline{\omega}\nu_0\}
\cap
\{\pm\overline{\omega}\,\overline{\nu_0}\}
=\emptyset$.
We obtain
\[
\{\pm\omega\nu_0,\pm\overline{\omega}\nu_0\}
\cap
\{\pm\omega\nu_0,
\pm\overline{\omega}\,\overline{\nu_0}\}
=
\{\pm\omega\nu_0\}.
\]
Hence
$\alpha_x'(\lambda)\in\{\pm\omega\nu_0\}
=\{\pm\lambda\}$.
This proves (i).
\end{proof}

We are in a position to prove that the mapping
$\alpha_x'$ behaves like identity or conjugation.

\begin{lem}\label{lem4.6}
For each $x\in\cha$, exactly one of the following alternatives holds:
\begin{enumerate}
\item[(i)] $\alpha_x'(\lambda)\in\{\pm\lambda\}$ for all $\lambda\in\T$;
\item[(ii)] $\alpha_x'(\lambda)\in\{\pm\overline{\lambda}\}$ for all $\lambda\in\T$.
\end{enumerate}
\end{lem}

\begin{proof}
Choose $\nu_0\in\T\setminus\{\pm1,\pm i\}$.
By Corollary~\ref{cor4.4},
$\alpha_x'(\nu_0)\in\{\pm\nu_0,\pm\overline{\nu_0}\}$.
If $\alpha_x'(\nu_0)\in\{\pm\nu_0\}$,
then (i) follows from Lemma~\ref{lem4.5}.
Otherwise,
$\alpha_x'(\nu_0)\in\{\pm\overline{\nu_0}\}$,
and (ii) follows from Lemma~\ref{lem4.5}.
\end{proof}

We split the Choquet boundary according to
Lemma~\ref{lem4.6}. Set
\[
X_0=
\{
x\in\cha:
\alpha_x'(\lambda)\in\{\pm\lambda\}
\text{ for all }\lambda\in\T
\},
\]
and
\[
X_1=
\{
x\in\cha:
\alpha_x'(\lambda)\in\{\pm\overline{\lambda}\}
\text{ for all }\lambda\in\T
\}.
\]
Then
\[
X_0\cup X_1=\cha,
\qquad
X_0\cap X_1=\emptyset.
\]
For $j\in\{0,1\}$ and $\lambda\in\T$, we write
\[
\lambda^{[j]}=
\begin{cases}
\lambda, & j=0,\\
\overline{\lambda}, & j=1.
\end{cases}
\]
Moreover, if $j\in\{0,1\}$ and $x\in X_j$, then
\begin{equation}\label{eq5.5}
\alpha_x'(\lambda)\in\{\pm\lambda^{[j]}\}
\qquad
(\lambda\in\T).
\end{equation}

We now determine the sign appearing in the two alternatives of
Lemma~\ref{lem4.6}.

For $\lambda,\nu\in\T$ and $x\in\cha$, it follows from
\eqref{eq4.1} that
\[
\alpha_x(\lambda)
=T(\lambda\unit)(\sigma(x)),
\qquad
\alpha_x(\nu)
=T(\nu\unit)(\sigma(x)).
\]
Since $\alpha_x'(\lambda)=
\overline{\alpha_x(1)}\alpha_x(\lambda)$
by \eqref{eq4.4},
we obtain
\begin{equation}\label{eq4.6}
|\alpha_x'(\lambda)\pm\alpha_x'(\nu)|
=|T(\lambda\unit)(\sigma(x))
\pm
T(\nu\unit)(\sigma(x))|
\end{equation}
with the same choice of sign on both sides.

The following lemma is a sign-comparison principle.  It will be used to
propagate equality of signs from one unimodular parameter to another.

\begin{lem}\label{lem4.7}
Let $j\in\{0,1\}$,
let $\lambda,\nu\in\T$ satisfy
$|\lambda+\nu|\ne|\lambda-\nu|$,
and let $x,p\in X_j$. If
$\alpha_x'(\nu)=\alpha_p'(\nu)$,
then
$\alpha_x'(\lambda)=\alpha_p'(\lambda)$.
\end{lem}

\begin{proof}
Since $x,p\in X_j$, inclusion~\eqref{eq5.5} gives
\[
\alpha_x'(\lambda),\alpha_p'(\lambda)\in\{\pm\lambda^{[j]}\},
\qquad
\alpha_x'(\nu),\alpha_p'(\nu)\in\{\pm\nu^{[j]}\}.
\]
By the hypothesis
$\alpha_x'(\nu)=\alpha_p'(\nu)$.
Then there exists $r\in\{\pm1\}$ such that
\[
\alpha_x'(\nu)=\alpha_p'(\nu)=r\nu^{[j]}.
\]
Suppose, to the contrary, that
$\alpha_x'(\lambda)\neq\alpha_p'(\lambda)$.
By interchanging $x$ and $p$ if necessary, we may assume that
\[
\alpha_x'(\lambda)=\lambda^{[j]},
\qquad
\alpha_p'(\lambda)=-\lambda^{[j]}.
\]
Using \eqref{eq4.6},
with the same choice of sign
throughout, we obtain from the point $x$
\[
|\lambda^{[j]}\pm r\nu^{[j]}|
\le
\|T(\lambda\unit)\pm T(\nu\unit)\|_\infty,
\]
and from the point $p$
\[
|{-\lambda}^{[j]}\mp r\nu^{[j]}|
\le
\|T(\lambda\unit)\mp T(\nu\unit)\|_\infty.
\]
Since complex conjugation preserves absolute values,
\[
|\la^{[j]}\pm r\nu^{[j]}|=|\la\pm r\nu|,\qq
|{-}\la^{[j]}\mp r\nu^{[j]}|=|\la\pm r\nu|,
\]
where the upper and lower signs are taken
simultaneously.
Since $r\in\set{\pm1}$,
both $|\lambda+\nu|$ and $|\lambda-\nu|$ are bounded above by
each of the two norms
\[
\|T(\lambda\unit)+T(\nu\unit)\|_\infty,
\qquad
\|T(\lambda\unit)-T(\nu\unit)\|_\infty.
\]
Therefore,
\[
\max\{|\lambda+\nu|,|\lambda-\nu|\}
\le
\min\{
\|T(\lambda\unit)+T(\nu\unit)\|_\infty,
\|T(\lambda\unit)-T(\nu\unit)\|_\infty
\}.
\]
Since $T$ is a phase-isometry,
the unordered pair of the two norms on
the right is
$\{|\lambda+\nu|,|\lambda-\nu|\}$.
Hence
\[
\max\{|\lambda+\nu|,|\lambda-\nu|\}
\le
\min\{|\lambda+\nu|,|\lambda-\nu|\},
\]
contrary to $|\lambda+\nu|\ne|\lambda-\nu|$.
Thus $\alpha_x'(\lambda)=\alpha_p'(\lambda)$.
\end{proof}

We now use the preceding comparison lemma to show that the sign
ambiguity in Lemma~\ref{lem4.6} is constant on each of the sets
$X_0$ and $X_1$.

\begin{lem}\label{lem4.8}
Let $j\in\{0,1\}$ and let $\lambda\in\T$.
If $X_j\ne\emptyset$, then there exists
$\eta_j(\lambda)\in\{\pm1\}$ such that
\[
\alpha_x'(\lambda)=\eta_j(\lambda)\lambda^{[j]}
\qquad (x\in X_j).
\]
\end{lem}

\begin{proof}
First assume that $\lambda\notin\{\pm i\}$.
Let $x,p\in X_j$. Since 
$|\lambda+1|\ne|\lambda-1|$ and
$\alpha_x'(1)=1=\alpha_p'(1)$,
Lemma~\ref{lem4.7}, applied with $\nu=1$, gives
\[
\alpha_x'(\lambda)=\alpha_p'(\lambda).
\]

It remains to treat $\lambda=\pm i$. Choose
$\nu_0=e^{\pi i/3}$.
Then $|i+\nu_0|\neq|i-\nu_0|$.
By the preceding result,
applied to $\lambda=\nu_0$,
where $\nu_0\notin\{\pm i\}$,
we obtain
$\alpha_x'(\nu_0)=\alpha_p'(\nu_0)$.
Lemma~\ref{lem4.7}, applied with
$(i,\nu_0)$ and $(-i,\nu_0)$, gives
$\alpha_x'(i)=\alpha_p'(i)$ and
$\alpha_x'(-i)=\alpha_p'(-i)$.

Therefore, for each $\lambda\in\T$,
the value $\alpha_x'(\lambda)$ is independent of
$x\in X_j$.

By \eqref{eq5.5},
$\alpha_x'(\lambda)\in\{\pm\lambda^{[j]}\}$
for each $x\in X_j$.
Then there exists
$\eta_j(\lambda)\in\{\pm1\}$ such that
$\alpha_x'(\lambda)
=\eta_j(\lambda)\lambda^{[j]}$
for all $x\in X_j$.
\end{proof}

Now we are ready to overcome the sign ambiguity
for $\alpha_x'$ appearing in Lemma~\ref{lem4.6}.

\begin{lem}\label{lem4.9}
For each $\lambda\in\T$, there exists
$\eta(\lambda)\in\{\pm1\}$ such that
\[
\alpha_x'(\lambda)=
\begin{cases}
\eta(\lambda)\lambda, & x\in X_0,\\[1mm]
\eta(\lambda)\overline{\lambda}, & x\in X_1.
\end{cases}
\]
\end{lem}

\begin{proof}
If either $X_0$ or $X_1$ is empty, the conclusion follows immediately
from Lemma~\ref{lem4.8}. Thus we may assume that both are nonempty.

Fix $\lambda\in\T$.
By Lemma~\ref{lem4.8}, there exist
$\eta_0(\lambda),\eta_1(\lambda)\in\{\pm1\}$ such that
\[
\alpha_x'(\lambda)=\eta_0(\lambda)\lambda
\qquad (x\in X_0)
\]
and
\[
\alpha_x'(\lambda)=\eta_1(\lambda)\overline{\lambda}
\qquad (x\in X_1).
\]
We prove that $\eta_0(\lambda)=\eta_1(\lambda)$.

First assume that $\lambda\notin\{\pm i\}$.
Suppose, to the contrary, that
$\eta_0(\lambda)\neq\eta_1(\lambda)$.
Thus
$\eta_1(\lambda)=-\eta_0(\lambda)$.
Choose $x\in X_0$ and $p\in X_1$.
It follows that
\[
\alxp(\la)=\eta_0(\la)\la,\qq
\alpha_p'(\la)=-\eta_0(\la)\ov{\la}.
\]
Since
$\alpha_x'(1)=\alpha_p'(1)=1$,
equation \eqref{eq4.6},
with the same choice of sign throughout, gives
\[
|\eta_0(\lambda)\lambda\pm1|
\le
\|T(\lambda\unit)\pm T(\unit)\|_\infty
\]
from the point $x$, and
\[
|{-\eta_0}(\lambda)\overline{\lambda}\mp1|
\le
\|T(\lambda\unit)\mp T(\unit)\|_\infty
\]
from the point $p$.
Since complex conjugation preserves modulus,
\[
|{-\eta_0}(\lambda)\overline{\lambda}\mp1|
=|\eta_0(\lambda)\lambda\pm1|,
\]
where the upper and lower signs are taken simultaneously. 
Hence, since $\eta_0(\la)\in\set{\pm1}$,
both $|\lambda+1|$ and $|\lambda-1|$ are bounded above by
each of the two norms
\[
\|T(\lambda\unit)+T(\unit)\|_\infty,
\qquad
\|T(\lambda\unit)-T(\unit)\|_\infty .
\]
Since $T$ is a phase-isometry,
\[
\max\{|\lambda+1|,|\lambda-1|\}
\le
\min\{|\lambda+1|,|\lambda-1|\},
\]
which contradicts
$\lambda\notin\{\pm i\}$.
Thus
$\eta_0(\lambda)=\eta_1(\lambda)$
for $\lambda\notin\{\pm i\}$.

It remains to treat $\lambda=\pm i$.
Choose $\lambda_0=e^{\pi i/3}$.
By the first part,
\[
\eta_0(\lambda_0)=\eta_1(\lambda_0).
\]
If $\eta_0(i)\ne\eta_1(i)$, then the preceding argument,
applied to the pair $(i,\lambda_0)$,
instead of $(\lambda,1)$, shows that both
$|i+\lambda_0|$
and
$|i-\lambda_0|$
are bounded above by each of the two norms
\[
\|T(i\unit)+T(\lambda_0\unit)\|_\infty,
\qquad
\|T(i\unit)-T(\lambda_0\unit)\|_\infty.
\]
Since $T$ is a phase-isometry,
the unordered pair formed by these two norms is precisely
\[
\{|i+\lambda_0|,|i-\lambda_0|\}.
\]
Hence each of
$|i+\lambda_0|$
and
$|i-\lambda_0|$
is bounded above by the smaller one, and therefore
\[
\max\{|i+\lambda_0|,|i-\lambda_0|\}
\le
\min\{|i+\lambda_0|,|i-\lambda_0|\},
\]
which is impossible.
The case $\lambda=-i$ is treated similarly.

Therefore we may define
\(\eta(\lambda)=\eta_0(\lambda)=\eta_1(\lambda)\),
and the assertion follows.
\end{proof}

\subsection{Normalization of local signs}

We have shown that the normalized maps
$\alpha_x'$
admit only the two canonical forms
corresponding to the identity and complex conjugation.
In this sense,
the phase behavior of the constant functions
$\lambda\unit\in\lambda M_x$
has been essentially determined.

In particular,
if the representation formulas
\eqref{main-eq} and \eqref{thm1.2}
in the Main Theorem are expected to hold,
then the image of the constant function $\unit$
should determine the unimodular weight
appearing in the representation formula
\eqref{main-eq}.
From this viewpoint,
the quantities $\alpha_x(1)$
may be regarded as local models
for the global phase factor.

We next normalize the remaining sign ambiguity
in the images of the maximal convex subsets
$\lambda M_x$.
For each $\lambda\in\T$,
let $\eta(\lambda)\in\{\pm1\}$
be the sign given by Lemma~\ref{lem4.9}.
For $x\in\cha$, define
\[
\widetilde{\alpha}_x(\lambda)
=
\eta(\lambda)\alpha_x(\lambda).
\]
Then, by \eqref{eq4.2},
\[
T(\pm\lambda M_x)
=
\pm\widetilde{\alpha}_x(\lambda)N_{\sigma(x)}
\qquad
(\lambda\in\T,\,x\in\cha).
\]

Thus the image of each maximal convex subset
$\lambda M_x$
has a canonical direction determined by
$\widetilde{\alpha}_x(\lambda)$.
Accordingly, for each
$f\in\lambda M_x$,
there exists a unique sign
$\theta_x(f)\in\{\pm1\}$
such that
\[
T(f)(\sigma(x))
=
\theta_x(f)\widetilde{\alpha}_x(\lambda).
\]
We regard $\theta_x(f)$
as the residual local sign ambiguity
of the element $f$ relative to the normalized direction
$\widetilde{\alpha}_x(\lambda)$.

Indeed, since
$T(f)\in
\pm\widetilde{\alpha}_x(\lambda)N_{\sigma(x)}$,
we have
$T(f)(\sigma(x))
\in\{\pm\widetilde{\alpha}_x(\lambda)\}$.
Therefore
\begin{equation}\label{eq4.7}
T(f)(\sigma(x))
=
\theta_x(f)\widetilde{\alpha}_x(\lambda)
\qquad
(f\in\lambda M_x).
\end{equation}

Since 
$\alpha_x'(\lambda)
=\overline{\alpha_x(1)}\alpha_x(\lambda)$
by \eqref{eq4.4},
we obtain
\[
\widetilde{\alpha}_x(\lambda)
=
\eta(\lambda)\alpha_x(1)\alpha_x'(\lambda).
\]
Hence Lemma~\ref{lem4.9} yields
\begin{equation}\label{eq4.8}
\widetilde{\alpha}_x(\lambda)
=
\begin{cases}
\alpha_x(1)\lambda,
& x\in X_0,
\\[1mm]
\alpha_x(1)\overline{\lambda},
& x\in X_1.
\end{cases}
\end{equation}

Finally, since
$T$ is odd and
$\widetilde{\alpha}_x(-\lambda)
=
-\widetilde{\alpha}_x(\lambda)$,
equality \eqref{eq4.7} gives
\begin{equation}\label{eq4.9}
\theta_x(-f)=\theta_x(f)
\qquad
(f\in\lambda M_x).
\end{equation}

\subsection{Propagation of local sign information}

The next proposition provides the analytic device
needed to compare local sign information
at different boundary points.
Its role is to construct functions
which simultaneously peak at two prescribed points
with independently controllable phases.
This allows the local phase behavior
associated with one maximal convex subset
to be transferred to another.

The proof uses a refined additive Bishop-type construction.
The refinement is essential because the resulting functions
must work simultaneously for all unimodular parameters
and preserve precise norm estimates.

\begin{prop}\label{prop4.10}
Let $x,p\in\cha$ with $x\ne p$, and let
$\varepsilon\in(0,1/4)$.
Then there exist
$h\in M_x$ and $k\in M_p$ such that
\[
h(p)=0,
\qquad
k(x)=0,
\]
and
\[
\|s\lambda h+t\nu k\|_\infty
=
\max\{s,t\}
\]
for all
$s,t\in[4\varepsilon,1]$
and all
$\lambda,\nu\in\T$.

In particular,
\[
\lamx\cap\nump\neq\emptyset
\qq(\la,\nu\in\T).
\]
\end{prop}

\begin{proof}
We give the details needed to make clear that the functions obtained
below are independent of the parameters $s,t$ and of the unimodular
constants.

Fix $\ve\in(0,1/4)$ and
disjoint open neighborhoods $U,V$ of $x,p$, respectively.
We first construct
families
\[
(f_n,g_n)\in M_x\times M_p
\]
and disjoint open neighborhoods
\[
(U_n,V_n)
\]
of $x,p$, respectively,
for $n\in\N$,
such that $U_1=U$, $V_1=V$, and for all
$s,t\in[4\varepsilon,1]$ and $\la,\nu\in\T$,
\[
f_n(p)=0,\qquad g_n(x)=0,
\]
and
\begin{equation}\label{eq4.10}
|s\la f_n+t\nu g_n|<
\begin{cases}
s+\varepsilon, & \text{on }U_n,\\
t+\varepsilon, & \text{on }V_n,\\
\varepsilon, & \text{on }X\setminus(U_n\cup V_n).
\end{cases}
\end{equation}

The initial functions $f_1$ and $g_1$ are obtained
from the peak property \eqref{peak}
and
Lemma~\ref{lem2.4}.
Indeed, there exist $f_1\in\mx$ and $g_1\in\Mp$ such that
\[
f_1(p)=0,\q
|f_1|<\fr{\ve}{2}
\qbox{on}\X\setminus U_1,\qq
g_1(x)=0,\q
|g_1|<\fr{\ve}{2}
\qbox{on}\X\setminus V_1.
\]
Then, since $s,t\le1$, the triangle inequality shows that
$f_1$ and $g_1$ satisfy \eqref{eq4.10} for $n=1$.
After $f_1,\ldots,f_n\in\mx$ and
$g_1,\ldots,g_n\in\Mp$ have been chosen, define
\[
U_{n+1}
=\bigcap_{j=1}^n
\left\{z\in U_n: |g_j(z)|<\fr{\varepsilon}{2^n}\right\},
\]
and
\[
V_{n+1}
=\bigcap_{j=1}^n
\left\{z\in V_n: |f_j(z)|<\fr{\varepsilon}{2^n}\right\}.
\]
These are disjoint open neighborhoods of $x$ and $p$.
Applying the preceding construction to
$U_{n+1}$ and $V_{n+1}$,
we obtain
$f_{n+1}\in M_x$
and
$g_{n+1}\in M_p$
with the same estimates as above.
This completes the induction.

Define
\[
h=\sum_{j=1}^{\infty}\fr{f_j}{2^j},
\qquad
k=\sum_{j=1}^{\infty}\fr{g_j}{2^j}.
\]
Then $h\in M_x$, $k\in M_p$, and
\[
h(p)=0,\qquad k(x)=0.
\]
Fix $s,t\in[4\varepsilon,1]$ and $\la,\nu\in\T$,
and put
$F=s\la h+t\nu k$.
Hence
\[
F=\sum_{j=1}^\infty\fr{s\la f_j+t\nu g_j}{2^j}.
\]
We prove $|F|\le s$ on $U=U_1$.  
Let $z\in U$.
First suppose that
$z\in\bigcap_{n=1}^{\infty} U_n$.
Fix $j\in\mathbb N$.
Since $z\in U_{n+1}$ for every $n\ge j$,
the defining property of $U_{n+1}$ gives
\[
|g_j(z)|<\frac{\varepsilon}{2^n}
\qquad (j\leq n).
\]
Letting $n\to\infty$,
we obtain $g_j(z)=0$. Hence
\[
s|f_j(z)|+t|g_j(z)|=s|f_j(z)|\le s.
\]
Since $j$ was arbitrary, we obtain
\[
|s\la f_j(z)+t\nu g_j(z)|\le s
\qquad (j\in\mathbb N).
\]
Therefore
\[
|F(z)|
\le
\sum_{j=1}^{\infty}
\frac{|s\la f_j(z)+t\nu g_j(z)|}{2^j}
\le s.
\]

If $z\in U\setminus\bigcap_{n=1}^\infty U_n$,
let $m$ be the first index such that
$z\notin U_{m+1}$.
Then $z\in U_m$, and hence
\[
s|f_j(z)|+t|g_j(z)|
<s+\fr{\varepsilon}{2^{m-1}}
\qquad(1\le j\le m-1).
\]
By \eqref{eq4.10},
\[
|s\la f_m(z)+t\nu g_m(z)|
<s+\varepsilon,
\]
and, for $j\ge m+1$,
\[
|s\la f_j(z)+t\nu g_j(z)|<\varepsilon,
\]
since
$z\not\in U_j\cup V_j$;
indeed, $z\in U_m\setminus U_{m+1}$,
$U_m\cap V_m=\emptyset$,
and the neighborhoods are decreasing.
Therefore
\[
\begin{aligned}
|F(z)|
&\le
\sum_{j=1}^{m-1}
\fr{1}{2^j}\left(s+\fr{\varepsilon}{2^{m-1}}\right)
+\fr{s+\varepsilon}{2^m}
+\sum_{j=m+1}^\infty\fr{\varepsilon}{2^j}\\
&=
\left(s+\fr{\ve}{2^{m-1}}\right)
\left(1-\fr{1}{2^{m-1}}\right)
+\fr{s+\ve}{2^m}+\fr{\ve}{2^m}\\
&=
s+\fr{1}{2^m}
\left(4\varepsilon-s-\fr{4\varepsilon}{2^m}\right)
\le s,
\end{aligned}
\]
because $4\varepsilon\le s$.
Thus $|F|\le s$ on $U$.
By a symmetric argument,
\[
|F|\le t
\qquad\text{on }V.
\]
On $X\setminus(U\cup V)$,
each summand satisfies
\[
|s\la f_j+t\nu g_j|<\varepsilon
\qquad
\text{on }X\setminus(U\cup V).
\]
Hence $|F|<\varepsilon$ there.
Consequently,
\[
|F|\le
\begin{cases}
s, & \text{on }U,\\
t, & \text{on }V,\\
\varepsilon, & \text{on }X\setminus(U\cup V).
\end{cases}
\]
Since $4\varepsilon\le s,t$, this gives
$\|F\|_\infty\le \max\{s,t\}$.
The reverse inequality follows from
$F(x)=s\la$ and $F(p)=t\nu$.
Therefore
$\|F\|_\infty=\max\{s,t\}$.

Taking $s=t=1$, we obtain
$\|\lambda h+\nu k\|_\infty=1$.
Since
$(\lambda h+\nu k)(x)=\lambda$
and $(\lambda h+\nu k)(p)=\nu$,
it follows that
$\lambda h+\nu k
\in\lambda M_x\cap \nu M_p$.
In particular,
$\lambda M_x\cap \nu M_p\ne\emptyset$
for $\lambda,\nu\in\T$.
\end{proof}

For $(\lambda,x)\in\T\times\cha$, define
\[
T_x(f)=\theta_x(f)T(f)\qquad (f\in\lambda M_x),
\]
where $\theta_x$ is defined by \eqref{eq4.7}.
Then
\[
T_x(f)\in \widetilde{\alpha}_x(\lambda)N_{\sigma(x)}
\qquad (f\in\lambda M_x).
\]

We prove that $T_x$ is an isometry
on $\lamx$.
This allows us to compare
the local sign functions
associated with different boundary points.

\begin{lem}\label{lem4.11}
Let $(\lambda,x)\in\T\times\cha$. Then
\[
\|T_x(f)-T_x(g)\|_\infty=\|f-g\|_\infty
\qquad (f,g\in\lambda M_x).
\]
\end{lem}

\begin{proof}
For $f,g\in\lambda M_x$, both $T_x(f)$ and $T_x(g)$ belong to
$\widetilde{\alpha}_x(\lambda)N_{\sigma(x)}$. Hence
$\|T_x(f)+T_x(g)\|_\infty=2$. Similarly, $\|f+g\|_\infty=2$.

The oddness of $T$ shows
$T_x(h)=T(\thx(h)h)$.
Thus, since
$\thx(f)\thx(g)\in\{\pm1\}$,
the phase-isometric property of $T$ gives
\[
\{\|T_x(f)+T_x(g)\|_\infty,\|T_x(f)-T_x(g)\|_\infty\}
=
\{\|f+g\|_\infty,\|f-g\|_\infty\}.
\]
The equality of the first entries implies the desired equality.
\end{proof}

We prove that the product $\theta_x\theta_p$
behaves consistently
at any two elements of $\lamx\cap\nump$.
This is the first step toward showing that
$\thx=\thp$ on $\lamx\cap\nump$.

\begin{lem}\label{lem4.12}
Let $\lambda,\nu\in\T$ and let $x,p\in\cha$ with $x\neq p$.
Then, for all $f,g\in\lambda M_x\cap\nu M_p$,
\[
\theta_x(f)\theta_p(f)=\theta_x(g)\theta_p(g).
\]
\end{lem}

\begin{proof}
Set
$C=\lambda M_x\cap\nu M_p$.
By Proposition~\ref{prop4.10}, the set $C$ is nonempty. It is also
convex, and hence connected.
By Lemma~\ref{lem4.11}, the map $T_x$ is continuous on $\lambda M_x$.
Thus $T_x(C)$ is connected.

For each $h\in C$, we have $h\in\nu M_p$, and hence
$T(h)(\sigma(p))
=\theta_p(h)\widetilde{\alpha}_p(\nu)$
by \eqref{eq4.7}.
Therefore
\[
T_x(h)(\sigma(p))
=
\theta_x(h)T(h)(\sigma(p))
=
\theta_x(h)\theta_p(h)\widetilde{\alpha}_p(\nu)
\qq(h\in C).
\]
Since $\theta_x(h)\theta_p(h)\in\{\pm1\}$,
this shows that
\[
T_x(C)
\subset
\widetilde{\alpha}_p(\nu)N_{\sigma(p)}
\cup
(-\widetilde{\alpha}_p(\nu)N_{\sigma(p)}).
\]

The two sets
$\widetilde{\alpha}_p(\nu)N_{\sigma(p)}$
and $-\widetilde{\alpha}_p(\nu)N_{\sigma(p)}$
are disjoint closed subsets of $\SB$. Since $T_x(C)$ is connected, it
must be contained entirely in one of them. Hence there exists
$r\in\{\pm1\}$ such that
$\Tx(C)\subset r\alxt(\nu)\nphp$.
This implies
\[
T_x(h)(\sigma(p))
=
r\,\widetilde{\alpha}_p(\nu)
\qquad (h\in C).
\]
Combining this with the equality above gives
\[
\theta_x(h)\theta_p(h)=r
\]
for all $h\in C$.
Thus $\theta_x(h)\theta_p(h)$ is constant on $C$, and the assertion
follows.
\end{proof}

We now extend the previous result
to arbitrary elements of
$\lambda M_x\cap\nu M_p$
and $\omega M_x\cap\nu M_p$.
This shows that local sign information
on $\lambda M_x\cap\nu M_p$
extends to
$\omega M_x\cap\nu M_p$
even when $\lambda\neq\omega$.

\begin{lem}\label{lem4.13}
Let $\lambda,\nu,\omega\in\T$ and let
$x,p\in\cha$ with $x\neq p$.
Then, for all
$f\in\lambda M_x\cap\nu M_p$ and
$g\in\omega M_x\cap\nu M_p$, we have
\[
\theta_x(f)\theta_p(f)=\theta_x(g)\theta_p(g).
\]
\end{lem}

\begin{proof}
First assume that
$\Re(\overline{\lambda}\omega)>0$.
By Proposition~\ref{prop4.10},
with $s=t=1$, we may choose
$h\in M_x$ and $k\in M_p$ such that
\[
f_1=\lambda h+\nu k\in\lambda M_x\cap\nu M_p,
\qquad
g_1=\omega h+\nu k\in\omega M_x\cap\nu M_p.
\]

We prove
$\theta_x(f_1)\theta_p(f_1)
=\theta_x(g_1)\theta_p(g_1)$.
Set
\[
a=\theta_x(f_1)\theta_p(f_1),\qq
b=\theta_x(g_1)\theta_p(g_1).
\]
Then $ab\in\set{\pm1}$.
Since $\Tp=\thp T$,
we obtain $aT_p(f_1)=T_x(f_1)$
and $bT_p(g_1)=T_x(g_1)$.

Equality \eqref{eq4.7} gives
$T_x(f_1)(\sigma(x))
=\widetilde{\alpha}_x(\lambda)$
and $T_x(g_1)(\sigma(x))
=\widetilde{\alpha}_x(\omega)$.
Using \eqref{eq4.8}, we obtain
\[
|\lambda+\omega|
=
|\widetilde{\alpha}_x(\lambda)+\widetilde{\alpha}_x(\omega)|
\le
\|T_x(f_1)+T_x(g_1)\|_\infty.
\]
Hence
\[
|\lambda+\omega|
\le
\Vinf{aT_p(f_1)+bT_p(g_1)}
=\|T_p(f_1)+ab\,T_p(g_1)\|_\infty.
\]
If $ab=-1$, then Lemma~\ref{lem4.11},
applied to $\nu M_p$, gives
\[
|\lambda+\omega|
\le
\|T_p(f_1)-T_p(g_1)\|_\infty
=
\|f_1-g_1\|_\infty
=
\|(\la-\omega)h\|_\infty
=
|\lambda-\omega|.
\]
This contradicts
$\Re(\overline{\lambda}\omega)>0$,
since this condition is equivalent to
$|1+\ov{\lambda}\omega|>|1-\ov{\lambda}\omega|$.
Therefore $ab=1$, and hence
$a=b$. Thus
$\theta_x(f_1)\theta_p(f_1)
=\theta_x(g_1)\theta_p(g_1)$.

Now let $\lambda,\omega\in\T$ be arbitrary.
Since any two points on the unit circle can be connected by a finite
chain with successive angular differences less than $\pi/2$,
we may choose
\[
\lambda=\lambda_0,\lambda_1,\ldots,\lambda_n=\omega
\]
such that
\[
\Re(\overline{\lambda_j}\lambda_{j+1})>0
\qquad
(0\le j\le n-1).
\]
Set $k_0=f$ and $k_n=g$.
For each $j$ with $1\leq j\leq n-1$, choose
$k_j\in\lambda_j M_x\cap\nu M_p$,
which is possible by Proposition~\ref{prop4.10}.
Applying the first part successively to the adjacent pairs
$(\lambda_j,\lambda_{j+1})$, we obtain
\[
\thx(k_j)\thp(k_j)
=\thx(k_{j+1})\thp(k_{j+1}),
\]
which yields
\[
\theta_x(f)\theta_p(f)
=\theta_x(k_0)\theta_p(k_0)
=\cdots
=\theta_x(k_n)\theta_p(k_n)
=
\theta_x(g)\theta_p(g).
\]
This proves the assertion.
\end{proof}

We are in a position to prove that
$\thx=\thp$ on $\lamx\cap\nump$.
This will be proved by comparing
sign information with the constant
function $\nu\unit\in\numx\cap\nump$,
whose image under $T$ has no
sign ambiguity.

\begin{lem}\label{lem4.14}
Let $\lambda,\nu\in\T$ and let
$x,p\in\cha$ with $x\neq p$.
Then
\[
\theta_x(f)=\theta_p(f)
\qquad
(f\in\lambda M_x\cap\nu M_p).
\]
\end{lem}

\begin{proof}
Fix
$f\in\lambda M_x\cap\nu M_p$.
Since $\nu\unit\in\nu M_x\cap\nu M_p$,
Lemma~\ref{lem4.13} shows
\[
\theta_x(f)\theta_p(f)
=
\theta_x(\nu\unit)\theta_p(\nu\unit).
\]
Formula \eqref{eq4.7} gives
\[
T(\nu\unit)(\sigma(x))
=
\theta_x(\nu\unit)\widetilde{\alpha}_x(\nu).
\]
On the other hand, by \eqref{eq4.1}
and since $\alxt(\la)=\eta(\la)\alx(\la)$,
\[
\widetilde{\alpha}_x(\nu)=\eta(\nu)T(\nu\unit)(\sigma(x)).
\]
Combining this with the preceding equality gives
$\thx(\nu\unit)\eta(\nu)=1$.
Similarly,
$\theta_p(\nu\unit)\eta(\nu)=1$.
Since $\eta(\nu)\in\set{\pm1}$,
\[
\theta_x(\nu\unit)
=
\eta(\nu)
=
\theta_p(\nu\unit).
\]
The preceding identity shows that
$\thx(f)\thp(f)=\thx(\nu\unit)\thp(\nu\unit)=1$,
and therefore
$\theta_x(f)=\theta_p(f)$.
\end{proof}

\subsection{Globalization and boundary representation}

We may now define a global sign function
\[
\theta\colon \SA\to\{\pm1\}
\]
as follows.
For each $f\in\SA$, choose $x\in\cha$ such that $|f(x)|=1$, and set
\[
\theta(f)=\theta_x(f).
\]
This is well-defined by Lemma~\ref{lem4.14}.
Moreover, \eqref{eq4.9} yields
\[
\theta(-f)=\theta(f)
\qquad (f\in\SA).
\]

By \eqref{eq4.7},
\[
T(f)(\sigma(x))
=
\theta(f)\widetilde{\alpha}_x(\lambda)
\qquad (f\in\lambda M_x).
\]
In particular, if $f\in\SA$ and $|f(x)|=1$, then
\[
T(f)(\sigma(x))
=
\theta(f)\widetilde{\alpha}_x(f(x)).
\]

By \eqref{eq4.8},
\[
\widetilde{\alpha}_x(\lambda)
=
\begin{cases}
\alpha_x(1)\lambda, & x\in X_0,\\[1mm]
\alpha_x(1)\overline{\lambda}, & x\in X_1.
\end{cases}
\]
Therefore, whenever $f\in S_A$ and $|f(x)|=1$, we have
\begin{equation}\label{eq4.11}
T(f)(\sigma(x))
=
\theta(f)\alpha_x(1)
\begin{cases}
f(x), & x\in X_0,\\[1mm]
\overline{f(x)}, & x\in X_1.
\end{cases}
\end{equation}

Recall that $\phi=\sigma^{-1}$ by Lemma~\ref{lem4.1}.
We define
\[
\gamma(y)=\alpha_{\phi(y)}(1)
\qquad (y\in\chb),
\]
and set
\[
Y_c=\phi^{-1}(X_0).
\]
Then $|\gamma|=1$ on $\chb$.
Writing $y=\sigma(x)$ in \eqref{eq4.11}, we obtain,
whenever $f\in S_A$ and $|f(\phi(y))|=1$,
\[
T(f)(y)
=
\theta(f)\gamma(y)
\begin{cases}
f(\phi(y)), & y\in Y_c,\\[1mm]
\overline{f(\phi(y))},
& y\in\chb\setminus Y_c.
\end{cases}
\]

This identity has been established only under
the condition $|f(\phi(y))|=1$.
To derive a genuine boundary representation, 
it remains to remove
this norm-attaining restriction.

For this purpose, we again use
Proposition~\ref{prop4.10}
to perturb a given function by suitable auxiliary functions while
preserving the prescribed boundary value.
The following auxiliary lemma allows us to transfer
the already established boundary phase information
from norm-attaining points
to arbitrary boundary points.

\begin{lem}\label{lem4.15}
Let $f\in S_A$, let $x,p\in\cha$ with $x\ne p$, and let
$\lambda,\nu\in\T$. Suppose that
\[
|f(x)|<1,\qquad
f(x)=|f(x)|\lambda,\qquad
f(p)=\nu.
\]
For each $\delta\in(0,1)$, put
\[
s=1-\delta|f(x)|,\qquad
t=1-\delta.
\]
Then there exist $h_\delta\in M_x$
and $k_\delta\in M_p$ such that
\[
h_\delta(p)=0,\qquad k_\delta(x)=0,\qq
\|s\lambda h_\delta
+t\nu k_\delta\|_\infty
=s,
\]
and
\[
g_\delta
=
\delta f
+s\lambda h_\delta
+t\nu k_\delta
\]
belongs to $\lambda M_x\cap\nu M_p$.

Moreover,
\[
\Vinf{g_\delta-f}
\leq2-\delta(1+|f(x)|).
\]
\end{lem}

\begin{proof}
Fix $\delta\in(0,1)$.
Since $|f(x)|<1$, we have $0<t<s$
and $f(x)\neq f(p)$.

Put
\[
d=\min\{|\delta f(x)-\delta f(p)|,\,1-\delta\}>0.
\]
For each $n\in\mathbb N$, define
\[
U_n=\left\{z\in X:
|\delta f(z)-\delta f(x)|<\fr{d}{2^{n+1}}\right\},
\]
and
\[
V_n=\left\{z\in X:
|\delta f(z)-\delta f(p)|<\fr{d}{2^{n+1}}\right\}.
\]
Then $U_n$ and $V_n$ are open neighborhoods of $x$ and $p$,
respectively. Since
$d\le |\delta f(x)-\delta f(p)|$, we have
$U_1\cap V_1=\emptyset$, and hence
$U_n\cap V_n=\emptyset$ for all $n\in\N$.

The construction used in the proof of
Proposition~\ref{prop4.10}
depends only on the ability to make the exterior estimate
arbitrarily small.
Repeating the same argument for these neighborhoods,
with exterior bound $t/2$,
we obtain functions
$f_n\in M_x$ and $g_n\in M_p$ for each $n\in\mathbb N$
such that
\[
f_n(p)=0,\qquad g_n(x)=0,
\]
and
\[
|s\la f_n+t\nu g_n|<
\begin{cases}
s, & \text{on } U_n,\\[1mm]
t, & \text{on } V_n,\\[1mm]
t/2, & \text{on } X\setminus(U_n\cup V_n).
\end{cases}
\]

Define
\[
h_\delta=\sum_{j=1}^{\infty}\fr{f_j}{2^j},
\qquad
k_\delta=\sum_{j=1}^{\infty}\fr{g_j}{2^j}.
\]
Then $h_\delta\in M_x$, $k_\delta\in M_p$, and
$h_\delta(p)=0,k_\delta(x)=0$.
Moreover,
\[
\|s\lambda h_\delta+t\nu k_\delta\|_\infty=s.
\]

Set
\[
g_\delta
=
\delta f+s\lambda h_\delta+t\nu k_\delta.
\]
Then
\[
g_\delta(x)
=
\delta |f(x)|\lambda+s\lambda
=\lambda,
\qquad
g_\delta(p)
=
\delta\nu+t\nu
=\nu.
\]
It remains to show that $\|g_\delta\|_\infty\le1$.

Let $z\in X$.  If $z\in\bigcap_{n=1}^{\infty}U_n$, then
$\delta f(z)=\delta f(x)$, and hence
\[
|g_\delta(z)|
\le
\delta |f(x)|
+\|s\la h_\delta+t\nu k_\delta\|_\infty
\le
\delta |f(x)|+s
=1.
\]
If $z\in U_1\setminus\bigcap_{n=1}^{\infty}U_n$,
let $m$ be the first
index such that $z\notin U_{m+1}$.
Then $z\in U_m$, and
\[
|\delta f(z)|
<\delta |f(x)|+\fr{d}{2^{m+1}}.
\]
Also,
\[
|s\la f_j(z)+t\nu g_j(z)|\le s\quad(1\le j\le m),\qquad
|s\la f_j(z)+t\nu g_j(z)|<\fr{t}{2}\quad(j\ge m+1),
\]
because
$z\notin U_j\cup V_j$
for all $j\ge m+1$.
Therefore
\[
\begin{aligned}
|g_\delta(z)|
&\le
\delta |f(z)|
+\sum_{j=1}^\infty\fr{|s\lambda f_j(z)+t\nu g_j(z)|}{2^j}\\
&\le
\delta|f(x)|+\fr{d}{2^{m+1}}
+\sum_{j=1}^{m}\fr{s}{2^j}
+\sum_{j=m+1}^{\infty}\fr{1}{2^j}\cdot\fr{t}{2} \\
&\leq
1-s+\fr{d}{2^{m+1}}+s\left(1-\fr{1}{2^m}\right)
+\fr{t}{2^{m+1}}\\
&=
1+
\fr{d+t-2s}{2^{m+1}}
\le 1,
\end{aligned}
\]
because $d\le t<s$.
Therefore $|g_\delta|\leq1$ on $U_1$.

The same argument on $V_1$ gives $|g_\delta(z)|\le1$ for all
$z\in V_1$.
Indeed, if $z\in\bigcap_{n=1}^{\infty}V_n$, then
$\delta f(z)=\delta f(p)$, and
\[
|g_\delta(z)|
\le
\delta
+\sum_{j=1}^{\infty}\fr{|s\la f_j(z)+t\nu g_j(z)|}{2^j}
\le
\delta+t
=1.
\]
If $z\in V_1\setminus\bigcap_{n=1}^{\infty}V_n$,
let $m$ be the first
index such that $z\notin V_{m+1}$. Then
$|s\la f_j(z)+t\nu g_j(z)|\le t$ for $1\le j\le m$,
while
$|s\la f_j(z)+t\nu g_j(z)|<t/2$ for $j\ge m+1$,
because
$z\notin U_j\cup V_j$
for all $j\ge m+1$.
Hence
\[
|g_\delta(z)|
\le
\delta
+\sum_{j=1}^{m}\fr{t}{2^j}
+\sum_{j=m+1}^{\infty}\fr{1}{2^j}\cdot\fr{t}{2}
\le 1.
\]

Finally, if $z\in X\setminus(U_1\cup V_1)$, then
$|s\la f_j(z)+t\nu g_j(z)|<t/2$ for every $j$, and so
\[
|g_\delta(z)|
\le
\delta+\sum_{j=1}^{\infty}\fr{1}{2^j}\cdot\fr{t}{2}
=
\delta+\fr{t}{2}
<1.
\]
Thus $\|g_\delta\|_\infty\le1$.  Since
$g_\delta(x)=\lambda$ and $g_\delta(p)=\nu$,
we conclude that
$g_\delta\in\lambda M_x\cap\nu M_p$.

Since
$\|s\lambda h_\delta+t\nu k_\delta\|_\infty
=s
=1-\delta|f(x)|$,
we obtain
\[
\|g_\delta-f\|_\infty
\le
\|(\delta-1)f\|_\infty
+\|s\lambda h_\delta+t\nu k_\delta\|_\infty
\le
2-\delta(1+|f(x)|).
\]
This completes the proof.
\end{proof}

We are ready to establish a partial representation
formula for the surjective phase-isometry $T$
on Choquet boundaries $\cha$ and $\chb$.

\begin{thm}\label{thm2}
There exist a homeomorphism
$\phi\colon\chb\to\cha$,
a continuous function $\gamma\colon\chb\to\T$,
a clopen subset
$Y_c$ of $\chb$, and a mapping
$\theta\colon\SA\to\{\pm1\}$
such that
\begin{equation}\label{eq4.12}
T(f)(y)
=
\theta(f)\gamma(y)
\begin{cases}
f(\phi(y)), & y\in Y_c,\\[1mm]
\overline{f(\phi(y))}, & y\in\chb\setminus Y_c,
\end{cases}
\end{equation}
for every $f\in\SA$ and every $y\in\chb$. Moreover,
\[
\theta(-f)=\theta(f)
\qquad (f\in\SA).
\]
\end{thm}

\begin{proof}
We have already constructed the maps
$\phi$, $\gamma$, and $\theta$,
and the set $Y_c$,
and established the representation formula
at norm-attaining points.

Fix $f\in S_A$ and $y\in\chb$, and put $x=\phi(y)$.
Assume that $|f(x)|<1$.
Since $f\in S_A$, there exist
$p\in\cha$ and $\nu\in\T$ such that
$f(p)=\nu$.
Choose $\lambda\in\T$ with
$f(x)=|f(x)|\lambda$.

For $0<\delta<1$,
let $g_\delta$ be as in Lemma~\ref{lem4.15}.
Since $f,g_\delta\in\nu M_p$,
Lemma~\ref{lem4.11},
together with the definition of the global sign
function $\theta$, gives
\[
\|\theta(g_\delta)T(g_\delta)
-\theta(f)T(f)\|_\infty
=\|g_\delta-f\|_\infty
\leq
2-\delta(1+|f(x)|).
\]
Moreover, $g_\delta\in\lambda M_x$,
so the norm-attaining formula \eqref{eq4.7}
gives
\[
\theta(g_\delta)T(g_\delta)(y)
=\widetilde{\alpha}_x(\lambda).
\]
Combining this with the preceding equality
shows
\[
1-|T(f)(y)|
\le
|\widetilde{\alpha}_x(\lambda)-\theta(f)T(f)(y)|
\le
\|g_\delta-f\|_\infty
\leq2-\delta(1+|f(x)|).
\]
Letting $\delta\to1-$ yields
\[
1-|T(f)(y)|\le 1-|f(x)|,
\]
and hence $|f(x)|\le |T(f)(y)|$. Applying the same argument to
$T^{-1}$ gives the reverse inequality. Thus
\begin{equation}\label{eq4.13}
|T(f)(y)|=|f(x)|.
\end{equation}
If $f(x)=0$, the desired formula
\eqref{eq4.12} is immediate.
Assume now that $f(x)\ne0$.

The preceding estimates also give, for all $0<\delta<1$,
\[
|\widetilde{\alpha}_x(\lambda)-\theta(f)T(f)(y)|
\le
2-\delta(1+|f(x)|).
\]
The preceding inequality shows that
\[
|\widetilde{\alpha}_x(\lambda)
-\theta(f)T(f)(y)|
\leq1-|f(x)|.
\]
Using \eqref{eq4.13}
and the triangle inequality,
\[
1=|\widetilde{\alpha}_x(\lambda)|
\le
|\widetilde{\alpha}_x(\lambda)-\theta(f)T(f)(y)|
+|\theta(f)T(f)(y)|
\leq1.
\]
Hence we obtain equality in the triangle inequality.
It follows that $\theta(f)T(f)(y)$
has the same argument as
$\widetilde{\alpha}_x(\lambda)$.
Since its modulus is $|f(x)|>0$,
we obtain
\[
\theta(f)T(f)(y)=|f(x)|\widetilde{\alpha}_x(\lambda).
\]
By \eqref{eq4.8} and the identity
$f(x)=|f(x)|\lambda$,
\[
|f(x)|\widetilde{\alpha}_x(\lambda)
=
\alpha_x(1)
\begin{cases}
f(x), & x\in X_0,\\[1mm]
\overline{f(x)}, & x\in X_1.
\end{cases}
\]
Since $x=\phi(y)$,
$\gamma(y)=\alpha_x(1)$, and
$Y_c=\phi^{-1}(X_0)$,
combining the preceding equalities gives
the desired representation formula
\eqref{eq4.12}.

Taking $f=\unit$ in \eqref{eq4.12} gives
$\gamma=\theta(\unit)T(\unit)$ on $\chb$, so $\gamma$ is continuous.
The identity $\theta(-f)=\theta(f)$ follows from the local identity
\eqref{eq4.9} and the definition of the global sign function.

We prove that $\phi$ is continuous. Let $O$ be open in $\cha$, take
$y_0\in\phi^{-1}(O)$, and put $x_0=\phi(y_0)$. By the
peak property \eqref{peak},
choose $f_0\in S_A$ such that $f_0(x_0)=1$ and
$|f_0|<1/3$ on $X\setminus O$.
Set
\[
W=\{y\in Y: |T(f_0)(y)|>2/3\}.
\]
Then $W$ is an open neighborhood of $y_0$,
since $|T(f_0)(y_0)|=|f_0(\phi(y_0))|=1$
by \eqref{eq4.13}. If $y\in W\cap\chb$, then
$|f_0(\phi(y))|=|T(f_0)(y)|>2/3$.
Since $|f_0|<1/3$ on $X\setminus O$,
we obtain $\phi(y)\in O$ for all $y\in W\cap\chb$.
Thus $W\cap\chb\subset\phi^{-1}(O)$, so $\phi$ is continuous. Applying the
same argument to $T^{-1}$ shows that $\phi^{-1}$ is continuous.

Finally, \eqref{eq4.12} gives
\[
\theta(i\unit)\overline{\gamma(y)}T(i\unit)(y)
=
\begin{cases}
\phantom{-}i, & y\in Y_c,\\[1mm]
-i, & y\in\chb\setminus Y_c.
\end{cases}
\]
Since $\gamma$ and $T(i\unit)$ are continuous on $\chb$, both $Y_c$
and its complement are closed in $\chb$. Thus $Y_c$ is clopen.
\end{proof}

We extend the partial representation formula of $T$
to Shilov boundaries and maximal ideal spaces.
The proof of the Main Theorem is now almost routine,
because $T$ is essentially a weighted composition
operator.

\begin{proof}[\textbf{Proof of Main Theorem}]
For $f\in A\setminus\{0\}$, put $f^*=f/\|f\|_\infty$.
By Theorem~\ref{thm2}, there exist a homeomorphism
$\phi\colon\chb\to\cha$, a continuous function
$\gamma\colon\chb\to\T$, a clopen subset $Y_c$ of $\chb$, and a
mapping $\theta\colon S_A\to\{\pm1\}$ such that
\begin{equation}\label{eq4.14}
T(f^*)(y)
=
\theta(f^*)\gamma(y)
\begin{cases}
f^*(\phi(y)), & y\in Y_c,\\[1mm]
\overline{f^*(\phi(y))}, & y\in\chb\setminus Y_c,
\end{cases}
\end{equation}
for all $f\in A\setminus\set{0}$
and $y\in\chb$, and $\theta(-f^*)=\theta(f^*)$.

Set $\Gamma=\theta(\unit)T(\unit)\in B$. Then
$\Gamma=\gamma$ on $\chb$.
Since $|\gamma|=1$ on $\chb$,
continuity and density give
$|\Gamma|=1$ on $\pb$.

Define
$\TT\colon A\to B|_{\pb}$ by
\begin{equation}\label{eq4.15}
\TT(f)=
\begin{cases}
\|f\|_\infty\theta(f^*)\overline{\Gamma}|_{\pb}\cdot T(f^*)|_{\pb}
&f\ne0,\\[1mm]
0, & f=0.
\end{cases}
\end{equation}
Then \eqref{eq4.14} gives
\begin{equation}\label{eq4.16}
\TT(f)(y)
=
\begin{cases}
f(\phi(y)), & y\in Y_c,\\[1mm]
\overline{f(\phi(y))}, & y\in\chb\setminus Y_c,
\end{cases}
\end{equation}
for all $f\in A$ and $y\in\chb$.

We first prove that $\TT$ is a real-linear algebra
isomorphism from $A$ onto $B|_{\pb}$.  Real-linearity and
multiplicativity follow from \eqref{eq4.16} on $\chb$, and hence on
$\pb$ by density.
If $\TT(f)=\TT(g)$,
then \eqref{eq4.16} implies that
$f(\phi(y))=g(\phi(y))$ for all $y\in\chb$,
and hence $f=g$ on $\cha$.
Since $\cha$ is a
boundary for $A$,
we conclude $f=g$.
Thus $\TT$ is injective.

Set $\mathcal{B}=\TT(A)$. Since $\TT$ is real-linear
and multiplicative, $\mathcal{B}$ is a real subalgebra of
$C(\pb)$. From \eqref{eq4.15} we have
$\Gamma\mathcal{B}\subset B|_{\pb}$.
Conversely, if $u\in B\setminus\{0\}$,
let $u^*=u/\|u\|_\infty$, and
choose $f\in S_A$ with
$T(f)=u^*$ and put $g=\|u\|_\infty\theta(f)f$.
Then $\|g\|_\infty=\|u\|_\infty$ and
$g^*=\theta(f)f$.
The oddness of $T$ gives
\[
T(g^*)=\theta(f)T(f)=\theta(f)u^*.
\]
Moreover, $\theta(g^*)=\theta(f)$
because $\th(-f)=\th(f)$.
Hence \eqref{eq4.15} gives
$\TT(g)
=\overline{\Gamma}\,u$
on $\pb$,
and therefore
$u=\Gamma\TT(g)\in\Gamma\mathcal{B}$.
Thus $B|_{\pb}\subset\Gamma\mathcal{B}$.
Therefore $B|_{\pb}=\Gamma\mathcal{B}$.

We now prove $\mathcal{B}=B|_{\pb}$.
Since $B|_{\pb}=\Gamma\mathcal{B}$, we obtain
$\unit\in\Gamma\mathcal{B}$.
Thus there exists
$v_1\in\mathcal{B}$ such that
\[
\Gamma v_1=\unit
\qquad\mbox{on $\pb$}.
\]
If $w\in\mathcal{B}$, then
\[
w=(\Gamma v_1)w=\Gamma(v_1w)
\qquad\mbox{on $\pb$},
\]
and $v_1w\in\mathcal{B}$
because $\mathcal{B}$ is an algebra.
Hence
$\mathcal{B}\subset\Gamma\mathcal{B}$.
Since
$\Gamma\mathcal{B}=B|_{\pb}$,
we obtain
$\mathcal{B}\subset B|_{\pb}$.
Conversely, let $u\in B|_{\pb}$.
Since $\Gam\in B$,
$\Gamma|_{\pb}\cdot u\in B|_{\pb}
=\Gamma\mathcal{B}$.
Then there exists $w_1\in\mathcal{B}$ such that
$\Gamma u=\Gamma w_1$
on $\pb$.
Since $\Gamma v_1=\unit$ on $\pb$,
\[
u=(\Gamma v_1)u=v_1(\Gamma u)
=v_1(\Gamma w_1)
=(\Gamma v_1)w_1=w_1.
\]
Hence $u=w_1\in\mathcal{B}$.
Thus
$B|_{\pb}\subset\mathcal{B}$,
and therefore
$\mathcal{B}=B|_{\pb}$.
This shows $\TT(A)=\mathcal{B}=B|_{\pb}$,
and hence $\TT\colon A\to B|_{\pb}$ is surjective.
In particular, $\Gamma|_{\pb}$ is invertible
in $B|_{\pb}$, since $\Gamma v_1=1$
on $\pb$ with $v_1\in\mathcal{B}=B|_{\pb}$.
Choose $G\in B$ with $G=v_1$ on $\pb$.
Then $(\Gamma G)|_{\pb}=\unit$.
Since $\pb$ is a boundary for $B$, $\Gamma G=\unit$ on $Y$;
therefore $\Gamma$ is invertible in $B$.
For $f\in S_A$, \eqref{eq4.15} gives
\[
T(f)|_{\pb}
=\theta(f)\Gamma|_{\pb}\cdot\TT(f).
\]
This proves \eqref{main-eq}.

We now construct a homeomorphism
between the maximal ideal spaces
of $A$ and $B$.
Let $\Delta\colon B\to B|_{\pb}$ be the restriction map.
Since $\pb$ is a boundary for $B$,
$\Del$ is injective.
Let $\Lambda\colon B|_{\pb}\to B$ be its inverse. Set
$\mathcal{T}=\Lambda\circ\TT\colon A\to B$.
Since $\TT$ is a real-linear algebra isomorphism,
so is $\mathcal{T}$.
By \cite[Theorem~2.1]{hatori2},
there exist a homeomorphism
$\Phi\colon\MB\to\MA$ and a clopen subset
$Y_m$ of $\MB$ such that
\begin{equation}\label{eq4.17}
\widehat{\mathcal{T}(f)}(\rho)
=
\begin{cases}
\widehat{f}(\Phi(\rho)), & \rho\in Y_m,\\[1mm]
\overline{\widehat{f}(\Phi(\rho))},
& \rho\in\MB\setminus Y_m,
\end{cases}
\end{equation}
for all $f\in A$ and $\rho\in\MB$.

For $y\in\pb$, let $\rho_y$ be evaluation at $y$; for
$x\in\pa$, let $\tau_x$ be evaluation at $x$. We first show that
\begin{equation}\label{eq4.18}
\Phi(\rho_y)=\tau_{\phi(y)}
\qquad (y\in\chb).
\end{equation}
Suppose that
\[
\Phi(\rho_y)\ne\tau_{\phi(y)}
\]
for some $y\in\chb$.
Since characters separate points of $A$, we may choose
$f\in A$ such that
\[
\widehat{f}(\Phi(\rho_y))\ne0,\qq
\widehat{f}(\tau_{\phi(y)})=0.
\]
Thus $f(\phi(y))=\tau_{\phi(y)}(f)=0$.
Then
\eqref{eq4.16} gives $\TT(f)(y)=0$. On the other hand, since
$\mathcal{T}(f)|_{\pb}=(\Del\circ\mathcal{T})(f)=\TT(f)$,
\[
\TT(f)(y)
=\mathcal{T}(f)(y)
=\rho_y(\mathcal{T}(f))
=\widehat{\mathcal{T}(f)}(\rho_y),
\]
which is either $\widehat{f}(\Phi(\rho_y))$ or its conjugate by
\eqref{eq4.17}; this is nonzero, a contradiction. Hence
\eqref{eq4.18} holds.

We extend the map $\phi\colon\chb\to\cha$ to
a homeomorphism between the Shilov boundaries
of $A$ and $B$.
Define
\[
\Psi_B\colon\pb\to\MB,
\quad \Psi_B(y)=\rho_y,\qq
\Psi_A\colon\pa\to\MA,
\quad \Psi_A(x)=\tau_x.
\]
These maps are continuous
with the weak-* topology and injective because
$A$ and $B$ separate points; compactness of the Shilov boundaries and
Hausdorffness of the maximal ideal spaces imply that they are
homeomorphisms onto their ranges.
From \eqref{eq4.18}
\[
\Phi(\Psi_B(y))=\Phi(\rho_y)
=\tau_{\phi(y)}=\Psi_A(\phi(y))
\qq(y\in\chb).
\]
Since $\phi(\chb)=\cha$, we obtain
$\Phi(\Psi_B(\chb))=\Psi_A(\cha)$. Since $\Phi$, $\Psi_A$, and
$\Psi_B$ are homeomorphisms onto their ranges, and since
$\chb$ and $\cha$ are dense in $\pb$ and $\pa$, respectively,
we obtain
\[
\Phi(\Psi_B(\pb))=\Psi_A(\pa).
\]
Thus for each $y\in\pb$ there is a unique
$\varphi(y)\in\pa$ such that
$\Phi(\Psi_B(y))=\Psi_A(\varphi(y))$.
It follows that
\[
\Phi(\rho_y)=\tau_{\varphi(y)}.
\]
Since $\varphi=
\Psi_A^{-1}\circ\Phi\circ\Psi_B$,
the map $\varphi$ is a homeomorphism
extending $\phi:\chb\to\cha$.

We prove the representation formula
\eqref{thm1.2} for $\TT$.
Define
\[
Y_s=\{y\in\pb:\rho_y\in Y_m\}=\Psi_B^{-1}(Y_m).
\]
Since $Y_m$ is clopen in $\MB$ and $\Psi_B$ is continuous,
$Y_s$ is clopen in $\pb$.
For $f\in A$ and $y\in\pb$, we have
from $\TT=\Del\circ\mathcal{T}$ that
\[
\TT(f)(y)=\widehat{\mathcal{T}(f)}(\rho_y).
\]
If $y\in Y_s$, then $\rho_y\in Y_m$, so \eqref{eq4.17} gives
\[
\TT(f)(y)=\widehat{f}(\Phi(\rho_y))
=\widehat{f}(\tau_{\varphi(y)})=f(\varphi(y)).
\]
The case $y\notin Y_s$ gives the conjugate formula.
Therefore \eqref{thm1.2} is proved.

Conversely, let
$T\colon S_A\to S_B$
be a mapping satisfying
\begin{equation}\label{eq4.19}
T(f)|_{\pb}
=
\theta(f)\Gamma|_{\pb}\cdot\TT(f)
\qquad (f\in S_A),
\end{equation}
where $\theta$, $\Gamma$, and $\TT$
satisfy the properties stated in the Main Theorem.

We first show that $T$ is a phase-isometry.
Let $f,g\in S_A$.
By \eqref{eq4.19}, since $|\Gamma|=1$ on
$\pb$, we have
\[
\|T(f)\pm T(g)\|_\infty
=\|\theta(f)\TT(f)\pm
  \theta(g)\TT(g)\|_\infty
=\|\TT(f)\pm \theta(f)\theta(g)\TT(g)\|_\infty
\]
with the same choice of signs.
Since $\theta(f)\theta(g)\in\set{\pm1}$,
\[
\{\|T(f)+T(g)\|_\infty,\|T(f)-T(g)\|_\infty\}
=
\{\|\TT(f)+\TT(g)\|_\infty,
  \|\TT(f)-\TT(g)\|_\infty\}.
\]
Since $\TT$ is of the form \eqref{thm1.2},
the last unordered pair is
$\{\|f+g\|_\infty,\|f-g\|_\infty\}$.
Thus $T$ is a phase-isometry.

By assumption,
$\theta(-f)=\theta(f)$
for all $f\in\SA$.
Since \(\TT\) is real-linear,
\eqref{eq4.19} yields
\[
T(-f)=-T(f)
\qquad (f\in S_A).
\]

We next prove that \(T\) is surjective.
Fix \(u\in S_B\).
Since \(\Gamma\) is invertible in $B$
and $\TT\colon A\to B|_{\pb}$ is surjective,
there exists \(g\in S_A\) such that
$\TT(g)
=\Gamma^{-1}u$
on \(\pb\).
By \eqref{eq4.19},
$T(g)|_{\pb}
=\theta(g)u|_{\pb}$.
Since \(\pb\) is a boundary for \(B\),
we obtain
$\theta(g)T(g)=u$.
Because $\theta(g)\in\{\pm1\}$,
it follows that
$T(\theta(g)g)=u$.
Thus $T$ is surjective.
\end{proof}

\section*{Acknowledgment}

The third author was supported by JST SPRING,
Grant Number JPMJSP2121.
The fourth author was supported by JSPS KAKENHI
Grant Number JP 25K07028.

\end{document}